\documentclass[final,1p,times]{elsarticle}
\usepackage{amsmath,amssymb,amsfonts,amsthm,graphicx, float, lineno}
\usepackage{caption}
\usepackage{hyperref}
\newtheorem{theorem}{Theorem}[section]
\newtheorem{proposition}[theorem]{Proposition}

\newtheorem{lemma}[theorem]{Lemma}
\newtheorem{remark}[theorem]{Remark}

\numberwithin{equation}{section}

\begin{document}

\begin{frontmatter}

\title{Hausdorff dimension of exceptional sets arising in $\theta$-expansions}

\author[upb,ismma]{Gabriela Ileana Sebe}
\ead{igsebe@yahoo.com}
\address[upb]{Faculty of Applied Sciences, Politehnica University of Bucharest, \\ Splaiul Independentei 313, 060042, Bucharest, Romania}
\address[ismma]{Gheorghe Mihoc-Caius Iacob Institute of Mathematical Statistics and Applied Mathematics of the Romanian Academy, Calea 13 Sept. 13, 050711 Bucharest, Romania}
\author[anmb]{Dan Lascu\corref{cor1}}
\ead{lascudan@gmail.com}
\address[anmb]{Mircea cel Batran Naval Academy, 1 Fulgerului, 900218 Constanta, Romania}
\cortext[cor1]{Corresponding author}

\begin{abstract}
For a fixed $\theta^2=1/m$, $m \in \mathbb{N}_+$, let $x \in [0, \theta)$ and 
$[a_1(x) \theta, a_2(x) \theta, \ldots]$ be the $\theta$-expansion of $x$. 
Our first goal is to extend for $\theta$-expansions the results of Jarnik \cite{J-1928} concerning the set of badly aproximable numbers and the set of irrationals whose partial quotients do not exceed a positive integer. 
Define 
$
L_n (x)= \displaystyle \max_{1 \leq i \leq n} a_i(x), x \in \Omega:=[0, \theta)\setminus \mathbb{Q} 
$.
The second goal is to complete our result inspired by Philipp \cite{Ph-1976} 
\[
\liminf_{n \to \infty} \frac{L_n(x) \log\log n}{n} = \frac{1}{\log \left( 1+ \theta^2\right)} \mbox{ for a.e. } x \in [0, \theta].
\]
In this regard we prove that for any $\eta > 0$ the set 
\[
E(\eta) = 
\left\lbrace 
x \in \Omega: \lim_{n \to \infty} \frac{L_n(x) \log\log n}{n} = \eta
\right\rbrace 
\]
is of full Hausdorff dimension. 

\end{abstract}
\begin{keyword}
$\theta$-expansions, partial quotients, Hausdorff dimension.
\end{keyword}

\end{frontmatter}

\sloppy

\section{Introduction}

The present paper continues our series of papers dedicated to $\theta$-expansions \cite{SL-2014,Sebe-2017,SL-2019}. 
Our aim here is to complete some results on extreme value theory obtained in \cite{SL-2023-1}. 
In order to do this, we introduce a powerful tool for discriminating between the sets of Lebesgue measure zero, namely the notion of Hausdorff dimension, developed by Hausdorff \cite{H-1919} in 1919. 

Fractional dimensional theory provides an indication of the size and complexity of a set and has applications in studying the exceptional sets arising in the metrical theory of continued fractions. 

Shortly thereafter Jarnik \cite{J-1928} applied it to number theoretical problems and published the first paper in which the investigation is inspired by a problem of Diophantine approximation. 
In fact, Jarnik determined the Hausdorff dimension of sets of real numbers very close to infinitely many rational numbers. 
He investigated the set of irrationals whose partial quotients are bounded, i.e., the set of badly approximable numbers from the point of view of Diophantine approximation, and the set of irrationals whose partial quotients do not exceed a positive integer. 

These results have been subsequently generalized in many directions. 

Our first goal is to extend for $\theta$-expansions the work of Jarnik which remains a reference in the theory of regular continued fractions (RCFs). 
Since the case $\theta=1$ refers to RCF-expansions, we generalize and even improve some results of Jarnik. 

For a fixed $\theta \in (0, 1)$, every $x \in \left(0, \theta \right)$ can be expanded into a finite or infinite $\theta$-\textit{expansion}
\begin{equation}
x = \frac{1}{\displaystyle a_1\theta
+\frac{1}{\displaystyle a_2\theta
+ \frac{1}{\displaystyle a_3\theta + \ddots} }} =: [a_1 \theta, a_2 \theta, a_3 \theta, \ldots]. \label{1.1}
\end{equation}
The positive integers $a_n$, $n \in \mathbb{N}_+:=\{1, 2, \ldots\}$, which are called \textit{partial quotients} or \textit{digits} are determined as follows.
Consider a generalization of the Gauss map $T_{\theta}: [0,\theta] \to [0,\theta]$,
\begin{equation}
T_{\theta}(x):=
\left\{
\begin{array}{ll}
{\displaystyle \frac{1}{x} - \theta \left \lfloor \frac{1}{x \theta} \right\rfloor} &
{\displaystyle \hbox{if } x \in (0, \theta],}\\
\\
0 & \hbox{if } x=0.
\end{array}
\right. \label{1.2}
\end{equation}
and 
$a_{n+1}(x) =  a_n\left(T_{\theta}(x)\right) = a_1\left(T^{n}_{\theta}(x)\right)$, $n \in \mathbb{N}_+$, where 
\begin{equation}
a_1(x) := \left\{\begin{array}{ll}
\left\lfloor \displaystyle \frac{1}{x \theta}\right\rfloor  & \hbox{if }  x \neq 0, \\
\\
\infty & \hbox{if }  x = 0
\end{array} \right. \label{1.3}
\end{equation}
Here $\left\lfloor \cdot \right\rfloor$ stands for integer part. 

This expansion introduced by Chakraborty and Rao \cite{CR-2003} has many of the usual properties of RCFs. Moreover, Chakraborty and Rao proved that the dynamical system given by the transformation $T_{\theta}$ admits an absolutely continuous invariant probability for certain values of $\theta$. 
They have identified for $\theta^2 = \displaystyle \frac{1}{m}$, $m \in \mathbb{N}_+$, the invariant measure for the transformation $T_{\theta}$ as 
\begin{equation} \label{1.4}
\mathrm{d}\gamma_{\theta} := \frac{1}{\log \left(1+\theta^{2}\right)}
\frac{\theta \,\mathrm{d}x}{1 + \theta x}.
\end{equation}

It was proved in \cite{CR-2003} that the dynamical system 
$([0,\theta], T_{\theta})$ is ergodic and the measure $\gamma_{\theta}$ is invariant under $T_{\theta}$, that is, $\gamma_{\theta} (A) = \gamma_{\theta} (T^{-1}_{\theta}(A))$
for any $A \in {\mathcal{B}}_{[0, \theta]}=$ the $\sigma$-algebra of all Borel subsets of $[0, \theta]$.

Moreover, if $\theta^2 = \frac{1}{m}$, $m \in \mathbb{N}_+$, $[a_1 \theta, a_2 \theta, a_3 \theta, \ldots]$ is the $\theta$-expansion of any $x \in (0, \theta)$ if and only if the following conditions hold: 
\begin{enumerate}
\item [(i)] 
$a_n \geq m$ for any $n \in \mathbb{N}_+$ 
\item [(ii)] 
in the case when $x$ has a finite expansion, i.e., $x = [a_1 \theta, a_2 \theta, \ldots,  a_n \theta]$, then $a_n \geq m+1$.
\end{enumerate}

Every irrational $x \in (0, \theta) \setminus \mathbb{Q}=: \Omega$ has an infinite $\theta$-expansion. Note that for all $n \in \mathbb{N}_+$, $a_n(x) \geq m$ and $T^n_{\theta}([a_1 \theta, a_2 \theta, \ldots]) = [a_{n+1} \theta, a_{n+2} \theta, \ldots]$. 

For all $n \in \mathbb{N}_+$, the finite truncation of (\ref{1.1}) 
\[
\frac{p_n(x)}{q_n(x)} = [a_1(x) \theta, a_2(x) \theta, \ldots, a_n(x) \theta]
\]
is called the $n$-\textit{th convergent} of the $\theta$-expansion of $x$. 
For every infinite $\theta$-expansion $[a_1 \theta, a_2 \theta, \ldots]$ the sequences $\{p_n\}_{n \geq -1}$ and $\{q_n\}_{n \geq -1}$ can be obtained by the following recursive relations 
\begin{eqnarray}
p_n(x) &=& a_n(x) \theta p_{n-1}(x) + p_{n-2}(x), \quad \label{1.5} \\
q_n(x) &=& a_n(x) \theta q_{n-1}(x) + q_{n-2}(x), \quad \label{1.6}
\end{eqnarray}
with $p_{-1}(x) := 1$, $p_0(x) := 0$, $q_{-1}(x) := 0$ and $q_{0}(x) := 1$. 
By induction, we obtain
\begin{equation}
p_{n-1}(x)q_{n}(x) - p_{n}(x)q_{n-1}(x) = (-1)^{n}, \quad n \in \mathbb{N}. \label{1.7}
\end{equation}
From (\ref{1.6}), we have that $q_n(x) \geq \theta$, $n \in \mathbb{N}_+$. Further, also from (\ref{1.6}) and by induction we get 
\begin{equation}
q_n(x) \geq  \left\lfloor \frac{n}{2} \right \rfloor \theta^2. \label{1.8} 
\end{equation}

Define 
$
L_n (x):= \displaystyle \max_{1 \leq i \leq n} a_i(x), x \in \Omega. 
$
Inspired by a Philipp's result \cite{Ph-1976} which answered a conjecture of Erd\"os, we have proved in \cite{SL-2023-1} that for a.e.  
$x \in [0, \theta]$ 
\[
\liminf_{n \to \infty} \frac{L_n(x) \log\log n}{n} = \frac{1}{\log \left( 1+ \theta^2\right)}.
\]

In 2002, Okano \cite{Okano} constructed some specific numbers $x \in [0, 1)$ and showed that for any $k \geq 2$
\[
\liminf_{n \to \infty} \frac{L_n(x) \log\log n}{n} = \frac{1}{\log k}.
\]
The results established by Philipp and Okano for RCF expansions were complemented by Wu and Xu \cite{Wu2009} by showing that its exceptional set is of full Hausdorff dimension. 

The second goal of this paper is to prove that for any $\eta \geq 0$, the set 
\begin{equation} \label{1.9}
E(\eta) = 
\left\lbrace 
x \in \Omega: \lim_{n \to \infty} \frac{L_n(x) \log\log n}{n} = \eta
\right\rbrace 
\end{equation}
is of full Hausdorff dimension. 

The paper is organized as follows. 
In Section 2, we make a brief survey of the Hausdorff dimension. 
In Section 3 we establish some basic metric properties of $\theta$-expansions.
In Section 4 we generalize the results obtained by Jarnik \cite{J-1928}. 
It is worth mentioning that Proposition \ref{Prop.Jarnik} is a key tool for our further results. 
We also add some concluding remarks. 
The final section is devoted to the proof of Theorem \ref{th.5.5}. In Section 5, using the lower and upper bounds of the Hausdorff dimension obtained in Section 4 for some sets, we achieve the second goal of the paper mentioned above.  

\section{Hausdorff dimension}

In this section we recall some definitions and we establish notations for later use.

Hausdorff measure is an extension of Lebesgue measure that allows for the measurement of subsets within $\mathbb{R}^n$ possessing dimensions smaller than $n$.
This includes subsets like submanifolds and the intriguing category of fractal sets.
Through the application of Hausdorff measure, it becomes possible to define the dimension of any set within $\mathbb{R}^n$ even in cases involving complex or intricate geometries.

Hausdorff's idea consists in measuring a set by covering it by an infinite countable family of sets of bounded diameter, and then in looking at what happens when the maximal diameter of these covering sets tends to $0$. 

For a non-empty set $E \subset \mathbb{R}$, its \textit{diameter}, denoted by $|E|$, is by definition
\[
|E| := \sup\left\lbrace |x-y|: \, x,y \in E \right\rbrace.
\]
Let $J$ be a finite or infinite set of indices. If for some positive real number $\delta$ the set $E$ and the collection $\{C_j\}_{j \in J}$ of subsets of $\mathbb{R}$ satisfy $E \subset \displaystyle \bigcup_{j \in J} C_j$ and $0 < |C_j| \leq \delta$ for any $j \in J$, then $\{C_j\}_{j \in J}$ is called a $\delta$\textit{-covering of }$E$.

Note that the sets in a countable cover can be any sets whatever. For example, they do not need to be open or closed. 

If $\mathcal{C} = \{C_j\}_{j \in J}$ is an infinite countable collection of sets in $\mathbb{R}$ and $s>0$ is a real number we say the $s$-\textit{total length of }$\mathcal{C}$ is 
\[
\Delta_s(\mathcal{C}) = \displaystyle \sum_{j \in J} |C_j|^s.
\]
For any positive real number $\delta$, we define the $s$-\textit{covered length} of $E$ as 
\[
H_{\delta, s} (E) := \displaystyle \inf_{J} \sum_{j \in J} |C_j|^s,
\]
where the infimum is taken over all the countable $\delta$-coverings $\{C_j\}_{j \in J}$ of $E$. 
Clearly, the function $\delta \mapsto H_{\delta, s} (E)$ is non-increasing. 
Consequently, 
\[
H_s (E) = \displaystyle \lim_{\delta \to 0} H_{\delta, s} (E) = \displaystyle \sup_{\delta > 0} H_{\delta, s} (E)
\]
is well-defined and lies in $[0, \infty]$. 

The \textit{Hausdorff dimension} of a set $E \subset \mathbb{R}$ denoted by $\dim_{H} (E)$, is the unique non-negative real number $s_0$ such that 
$H_s (E) = 0$ if $s>s_0$ and $H_s (E) = +\infty$ if $0<s<s_0$.
In other words, we have 
\begin{eqnarray*}
\dim_{H}(E) &=& \inf\{s: \, H_s(E)=0\} \\
&=& \sup\{s: \, H_s(E)=+\infty\}. 
\end{eqnarray*}

Recall some properties of Hausdorff dimension for subsets $E, E_1, E_2, \ldots$ of $\mathbb{R}$: 
\begin{itemize}
\item[(i)]
If $E_1 \subset E_2$, then $\dim_{H}(E_1) \leq \dim_{H}(E_2)$;

\item[(ii)] 
$\displaystyle\dim_H \left( \bigcup^{\infty}_{j=1} E_j\right)  = \sup\{ \dim_H (E_j): j \geq 1\}$;

\item[(iii)] The Hausdorff dimension of a finite or countable set of points is $0$;

\item[(iv)] Two sets differing by a countable set of points have the same Hausdorff dimension. 
\end{itemize}
We also present the following result (see \cite[Proposition 2.3]{Falconer}). 
\begin{lemma} \label{lema.2.1}
Given a function $h: E \subset [0, 1) \to [0, 1)$ and suppose that it satisfies the following $\mu$-H\"older condition ($\mu > 0$), for some constant $C > 0$
\[
\left| h(x) - h(y) \right|  \leq C 
\left| x - y \right|^{\mu}, \mbox{ for all } x ,y \in E. 
\]
Then, 
\[
\dim_H(h(E)) \leq \frac{1}{\mu} \dim_H(E). 
\]
\end{lemma}
For more details and for an extensive exposure of the Hausdorff measure and the properties of the Hausdorff dimension we recommend Falconer's book \cite{Falconer}. 

\section{Some basic metric properties of $\theta$-expansions}

Let us fix $\theta^2 = \displaystyle\frac{1}{m}$, $m \in \mathbb{N}_+$. 
Putting $\mathbb{N}_m = \{m, m+1, \ldots\}$, $m \in \mathbb{N}_+$, the partial quotients $a_n$, $n \in \mathbb{N}_+$, take positive integer values in $\mathbb{N}_m$. 

For any $n \in \mathbb{N}_+$ and 
$(a_1, \ldots, a_n) \in \mathbb{N}_m^n$, let 
\begin{equation*}
I_n \left( a_1, \ldots, a_n \right) = \{x \in \Omega:  a_1(x) = a_1, \ldots, a_n(x) = a_n \}
\end{equation*}
be the $n$-th order fundamental interval. For $\theta$-expansions, such intervals generate the most natural partition of the interval $[0, \theta]$. 

From the definition of $T_{\theta}$, (\ref{1.5}) and (\ref{1.6}) we have for any $n \in \mathbb{N}_+$ and $(a_1, \ldots, a_n) \in \mathbb{N}_m^n$, 
\begin{equation}
I_n(a_1, \ldots, a_n) = \left\{
\begin{array}{lll}
	\left[ \displaystyle \frac{p_n}{q_n}, \frac{p_n+ \theta p_{n-1}}{q_n+ \theta q_{n-1}} \right)  & \quad \mbox{if $n$ is even}, \\
	\\
	\left(\displaystyle \frac{p_n+ \theta p_{n-1}}{q_n+ \theta q_{n-1}}, \frac{p_n}{q_n} \right] & \quad \mbox{if $n$ is odd}. \\
\end{array}
\right. \label{3.1}
\end{equation}
Using (\ref{1.7}) we get 
\begin{eqnarray}
\left|I_n\left(a_1, \ldots, a_n\right)\right| 
= \frac{\theta}{q_n (q_n + \theta q_{n-1})} \label{3.2}
\end{eqnarray}
and 
\begin{equation} \label{3.03}
\frac{\theta}{ \left( 1+ \theta^2 \right) q^2_n } \leq 
\left|I_n\left(a_1, \ldots, a_n\right)\right| 
\leq 
\frac{\theta}{q^2_n}.
\end{equation}

By (\ref{3.1}), the endpoints of the interval 
$I_{n+1}(a_1, \ldots, a_n, k)$, $k \geq m$, are $\displaystyle\frac{p_{n+1}}{q_{n+1}}$ and $\displaystyle\frac{p_{n+1}+ \theta p_{n}}{q_{n+1}+\theta q_{n}}$ with $p_{n+1} = k \theta p_n + p_{n-1}$ and $q_{n+1} = k \theta q_n + q_{n-1}$.
So we obtain 
\begin{equation*}
\frac{p_{n+1}}{q_{n+1}} = \frac{k \theta p_n + p_{n-1}}{k \theta q_n + q_{n-1}}, \quad
\frac{p_{n+1}+\theta p_n}{q_{n+1}+\theta q_n} = \frac{(k+1) \theta p_n + p_{n-1}}{(k+1) \theta q_n + q_{n-1}}
\end{equation*}
and
\begin{equation}
\left|I_{n+1}\left(a_1, \ldots, a_n, k\right)\right| 
= \frac{\theta}{(k\theta q_n + q_{n-1})((k+1)\theta q_n + q_{n-1})}.  \label{3.3}
\end{equation}

\begin{lemma} \label{lema.3.01}
For any $n \in \mathbb{N}_+$ and $(a_1, \ldots, a_n) \in \mathbb{N}^n_m$, we have
\begin{itemize}
\item[(i)]  $q_n(a_1, \ldots, a_n) \geq (m+1)^{\frac{n-1}{2}}$;
\item[(ii)] $\displaystyle \frac{(a_k+m)\theta}{2} \leq \frac{q_n(a_1, \ldots, a_n)}{q_{n-1}(a_1, \ldots, a_{k-1}, a_{k+1}, \ldots, a_n)} \leq (a_k+m) \theta$, $1 \leq k \leq n$.
\end{itemize}
\end{lemma}
\begin{proof}
$(i)$
Since by (\ref{1.6}) for any $n \geq 2$ we have $q_{n-1} \geq \frac{q_{n-2}}{\theta}$ it follows that 
\[
q_n = a_n \theta q_{n-1} + q_{n-2} \geq a_n \theta \frac{q_{n-2}}{\theta} + q_{n-2} = (a_n+1)q_{n-2} \geq (m+1) q_{n-2}. 
\]
The successive application of this inequality gives us
\begin{eqnarray*}
&&q_{2n} \geq (m+1)^n q_0 = (m+1)^n, \\
&&q_{2n+1} \geq (m+1)^n q_1 = (m+1)^n a_1 \theta \geq (m+1)^n m \theta \geq (m+1)^n. 
\end{eqnarray*}

$(ii)$
For any fixed $n \in \mathbb{N}_+$, we prove this result by induction on $n-k$. 
When $n-k=0$, by (\ref{1.6}), 
\[
\frac{q_k(a_1, \ldots, a_k)}{q_{k-1}(a_1, \ldots, a_{k-1})} = 
\frac{a_k \theta q_{k-1}(a_1, \ldots, a_{k-1}) + q_{k-2}(a_1, \ldots, a_{k-2})}{q_{k-1}(a_1, \ldots, a_{k-1})}  
\]
Since $q_{k-2} \leq {\theta} q_{k-1}$, 
\[
\frac{(a_k+m)\theta}{2} \leq \frac{q_k(a_1, \ldots, a_k)}{q_{k-1}(a_1, \ldots, a_{k-1})} \leq (a_k+m) \theta.
\]


%
Now, since $a_{k+1} \geq m$ and $q_{k-2} \leq {\theta} q_{k-1}$,
\begin{eqnarray*}
\frac{q_{k+1}(a_1, \ldots, a_{k+1})}{q_{k}(a_1, \ldots, a_{k-1}, a_{k+1})} &=& 
\frac{a_{k+1} \theta q_{k}(a_1, \ldots, a_{k}) + q_{k-1}(a_1, \ldots, a_{k-1})}{a_{k+1} \theta q_{k-1}(a_1, \ldots, a_{k-1}) + q_{k-2}(a_1, \ldots, a_{k-2})} \\
&=& \frac{(a_{k+1}a_k \theta^2 +1 )q_{k-1}(a_1, \ldots, a_{k-1}) + a_{k+1} \theta q_{k-2}(a_1, \ldots, a_{k-2})}{a_{k+1} \theta q_{k-1}(a_1, \ldots, a_{k-1}) + q_{k-2}(a_1, \ldots, a_{k-2})} \\
&\leq& a_k \theta + 
\frac{q_{k-1}(a_1, \ldots, a_{k-1}) + a_{k+1} \theta q_{k-2}(a_1, \ldots, a_{k-2})}{a_{k+1} \theta q_{k-1}(a_1, \ldots, a_{k-1})} \\
&=& a_k \theta + \frac{1}{a_{k+1}\theta} + \frac{q_{k-2}}{q_{k-1}}\leq (a_k+m) \theta,
\end{eqnarray*}
\begin{eqnarray*}
\frac{q_{k+1}(a_1, \ldots, a_{k+1})}{q_{k}(a_1, \ldots, a_{k-1}, a_{k+1})} &=& 
\frac{(a_{k+1}a_k \theta^2 +1 )q_{k-1}(a_1, \ldots, a_{k-1}) + a_{k+1} \theta q_{k-2}(a_1, \ldots, a_{k-2})}{a_{k+1} \theta q_{k-1}(a_1, \ldots, a_{k-1}) + q_{k-2}(a_1, \ldots, a_{k-2})} \\
&\geq& 
\frac{(a_{k+1}a_k \theta^2 +1 )q_{k-1}(a_1, \ldots, a_{k-1})}{a_{k+1} \theta q_{k-1}(a_1, \ldots, a_{k-1}) + \theta q_{k-1}(a_1, \ldots, a_{k-1})} \\
&=& \frac{a_{k+1}a_k \theta^2 +1 }{(a_{k+1} +1) \theta } \geq \frac{(a_k+m)\theta}{2}.
\end{eqnarray*}
Using (\ref{1.6}) and applying an induction argument, we get the desired result. 
\end{proof}

\section{Improvement and generalization of a Jarnik result}

For any $M \in \mathbb{N}_m$ and $n \in \mathbb{N}_+$, let 
\[
E_M :=\left\lbrace x \in \Omega: \, m \leq a_n(x) \leq M \mbox{ for any } n \geq 1 \right\rbrace 
\]
and
\[
E^n_M :=\left\lbrace x \in \Omega: \, m \leq a_i(x) \leq M \mbox{ for } i = 1, \ldots, n \right\rbrace. 
\]
Clearly we have for any $n \in \mathbb{N}_+$ 
\begin{equation}
E^n_M \supset E^{n+1}_M, \quad E_M = \displaystyle \bigcap_{n \geq 1} E^n_M \label{3.4}
\end{equation}
\begin{equation}
E^n_M = \sum_{ \begin{array}{ll}
m \leq a_i \leq M \\
i=1, \ldots, n
\end{array}
} 
I_n(a_1, \ldots, a_n). \label{3.5}
\end{equation}

The exact calculation of the Hausdorff dimension of a set is in most cases a difficult problem. It is however often possible to obtain lower and upper bounds for the Hausdorff dimension. 
In the sequel we shall obtain a result that improves Proposition 4 obtained by Jarnik \cite{J-1928}. 
In the case of RCF-expansion ($\theta = 1$), Jarnik established that for any $M > 8$
\begin{equation}
1-\frac{4}{M \log 2} \leq \dim_H (E_M) \leq 1 - \frac{1}{8M\log M}. \label{3.6} 
\end{equation}

\subsection{The lower bound}
The following lemma allows us to bound from below the Hausdorff dimension of $E_M$.

\begin{lemma} \label{lema.3.1}
Let us fix a real number $s \in (0, 1)$ and a positive integer $M \geq m$. 
If for any $n \geq 1$ and for all digits $a_1, \ldots, a_{n-1}$ ($m \leq a_i \leq M$, $i =1, \ldots, n-1$) the following inequality holds
\begin{equation}
\left| I_{n-1} (a_1, \ldots, a_{n-1}) \right|^s \leq \sum_{k=m}^{M} \left| I_{n} (a_1, \ldots, a_{n-1}, k) \right|^s, \label{3.7} 
\end{equation}
then $\dim_H (E_M) \geq s$. 
\end{lemma}

\begin{proof}
By (\ref{3.5}) and (\ref{3.4}) for any $n \geq 1$ all the $n$-th order intervals $I_n$ form a $\delta$-covering $\mathcal{C}_n$ for $E^n_M$ and implicitly for $E_M$. 
Any interval $I_n$ belonging to $\mathcal{C}_n$ is contained in exactly one of the intervals of $\mathcal{C}_{n-1}$ and contains at least two intervals belonging to $\mathcal{C}_{n+1}$. 
The maximum of the lengths of the intervals in $\mathcal{C}_n$ tends to $0$ when $n$ tends to infinity. 

From (\ref{3.7}) we get 
\[
\theta^s = \Delta_s(\mathcal{C}_0) \leq \ldots \leq \Delta_s(\mathcal{C}_{n-1}) \leq \Delta_s(\mathcal{C}_n) \leq \ldots
\]
where $\mathcal{C}_0 = \{ I_0 \}$ with $I_0 = [0, \theta]$. 

We deduce that the $s$-covered length of $E_M$, $H_{\delta', s} (E_M) \geq \theta^s$ for any $\delta' \leq \delta $. 
By letting $\delta'$ tend to $0$ we have 
\[
H_s(E_M) = \lim_{\delta' \to 0} H_{\delta', s} (E_M) \geq \theta^s >0,
\] 
which implies that the Hausdorff dimension of $E_M$ is at least equal to $s$.
\end{proof}
\begin{proposition} \label{prop.3.2}
For $s = 1 - \displaystyle \frac{2(m+1)}{M+1} \frac{1}{\log(m+1)}$, where $M > 2m+1$, the Hausdorff dimension $\dim_H(E_M) \geq s$. 
\end{proposition}
\begin{proof}
We check if the assumption from Lemma \ref{lema.3.1} is fulfilled. The inequality to be proved is (\ref{3.7}). 
Using (\ref{3.2}) and (\ref{3.3}) we obtain 
\begin{equation} \label{3.8}
\frac{\theta^s}{q^s_{n-1} (q_{n-1} + \theta q_{n-2})^s} \leq 
\sum_{k=m}^{M} \frac{\theta^s}{(k \theta q_{n-1} + q_{n-2})^s  \,  ((k+1) \theta q_{n-1} + q_{n-2})^s}
\end{equation}
Using (\ref{1.7}) we get 
\begin{eqnarray} \label{3.9}
&&\sum_{k=m}^{M} \frac{1}{(k \theta q_{n-1} + q_{n-2})  \,  ((k+1) \theta q_{n-1} + q_{n-2})} 
= \frac{(-1)^{n-1}}{\theta} \sum_{k=m}^{M} \left( \frac{k \theta p_{n-1} + p_{n-2}}{k \theta q_{n-1} + q_{n-2}} - \frac{(k+1) \theta p_{n-1} + p_{n-2}}{(k+1) \theta q_{n-1} + q_{n-2}} \right) \nonumber \\
&&= \frac{(-1)^{n-1}}{\theta} 
\left( \frac{m \theta p_{n-1} + p_{n-2}}{m \theta q_{n-1} + q_{n-2}} - \frac{(M+1) \theta p_{n-1} + p_{n-2}}{(M+1) \theta q_{n-1} + q_{n-2}} \right) \nonumber  \\
&&= \frac{(-1)^{n-1}}{\theta} 
\left[  \left( \frac{m \theta p_{n-1} + p_{n-2}}{m \theta q_{n-1} + q_{n-2}} - \frac{p_{n-1}}{q_{n-1}} \right) + \left( \frac{p_{n-1}}{q_{n-1}} -  \frac{(M+1) \theta p_{n-1} + p_{n-2}}{(M+1) \theta q_{n-1} + q_{n-2}} \right) \right]  \nonumber  \\
&&= \frac{1}{\theta} 
\left( \frac{1}{q_{n-1}(m \theta q_{n-1} + q_{n-2})} - \frac{1}{q_{n-1}((M+1) \theta q_{n-1} + q_{n-2})} \right). 
\end{eqnarray}
Using (\ref{3.9}) directly computation yields that 
\begin{eqnarray*}
&&\sum_{k=m}^{M} \frac{1}{(k \theta q_{n-1} + q_{n-2})^s  \,  ((k+1) \theta q_{n-1} + q_{n-2})^s} = 
\sum_{k=m}^{M} \frac{(k \theta q_{n-1} + q_{n-2})^{1-s} \,  ((k+1) \theta q_{n-1} + q_{n-2})^{1-s}}{(k \theta q_{n-1} + q_{n-2})  \,  ((k+1) \theta q_{n-1} + q_{n-2})} \\
&&\geq 
\sum_{k=m}^{M} \frac{(m \theta q_{n-1} + q_{n-2})^{1-s} \,  ((m+1) \theta q_{n-1}+q_{n-2})^{1-s}}{(k \theta q_{n-1} + q_{n-2})  \,  ((k+1) \theta q_{n-1} + q_{n-2})} \\
&&\geq 
\sum_{k=m}^{M} \frac{(m \theta q_{n-1} + q_{n-2})^{1-s} \,  ((m+1) \theta q_{n-1})^{1-s}}{(k \theta q_{n-1} + q_{n-2})  \,  ((k+1) \theta q_{n-1} + q_{n-2})} \\
&&= (m \theta q_{n-1} + q_{n-2})^{1-s} \,  ((m+1) \theta q_{n-1})^{1-s} \frac{1}{\theta} \frac{1}{q_{n-1}(m \theta q_{n-1} + q_{n-2})} \left[ 1- \frac{m \theta q_{n-1} + q_{n-2}}{(M+1) \theta q_{n-1} + q_{n-2}} \right] \\
&&= \frac{(m+1)^{1-s}}{(\theta q_{n-1})^s \, (m \theta q_{n-1} + q_{n-2})^s} \left(  1- \frac{m \theta q_{n-1} + q_{n-2}}{(M+1) \theta q_{n-1} + q_{n-2}} \right) \\
&&=\frac{(m+1)^{1-s}}{(\theta q_{n-1})^s} \frac{\theta^s}{(q_{n-1} + \theta q_{n-2})^s} \left( 1- \frac{m \theta q_{n-1} + q_{n-2}}{(M+1) \theta q_{n-1} + q_{n-2}} \right) \\
&&= \frac{(m+1)^{1-s}}{( q_{n-1})^s \, (q_{n-1} + \theta q_{n-2})^s } \left( 1- \frac{ q_{n-1} + \theta q_{n-2}}{(M+1) \theta^2 q_{n-1} + \theta q_{n-2}} \right).
\end{eqnarray*}
To obtain (\ref{3.8}) we have to show that 
\begin{equation} \label{3.10}
(m+1)^{1-s} \left( 1- \frac{ q_{n-1} + \theta q_{n-2}}{(M+1) \theta^2 q_{n-1} + \theta q_{n-2}} \right) \geq 1.
\end{equation}
Since $q_{n-2} \leq \theta q_{n-1}$ we get
\[
\frac{ q_{n-1} + \theta q_{n-2}}{(M+1) \theta^2 q_{n-1} + \theta q_{n-2}} \leq \frac{ q_{n-1} + \theta^2 q_{n-1}}{(M+1) \theta^2 q_{n-1}} = \frac{1+\theta^2}{(M+1) \theta^2} = \frac{m+1}{M+1},
\]
and 
\[
(m+1)^{1-s} \left( 1- \frac{ q_{n-1} + \theta q_{n-2}}{(M+1) \theta^2 q_{n-1} + \theta q_{n-2}} \right) \geq  (m+1)^{1-s} \left( 1 - \frac{m+1}{M+1} \right).
\]
Clearly 
\[
(m+1)^{1-s} \left( 1 - \frac{m+1}{M+1} \right) \geq 1 
\]
if and only if 
\[
(1-s) \log(m+1) \geq - \log \left( 1 - \frac{m+1}{M+1} \right). 
\]
Since $2x \geq - \log(1-x)$, for all $x \in (0, 1/2)$, by choosing 
$(1-s) \log(m+1) = 2 \displaystyle \frac{m+1}{M+1}$,
and assuming that $\displaystyle \frac{m+1}{M+1} \in \left(0, \frac{1}{2}\right)$, we obtain (\ref{3.10}) for $s = 1 - \displaystyle \frac{2(m+1)}{M+1} \frac{1}{\log(m+1)}$, with $M > 2m+1$.

\end{proof}

\subsection{The upper bound}
Now we bound from above the Hausdorff dimension of $E_M$ by applying the following lemma. 
\begin{lemma} \label{lema.3.3}
Let us fix a real number $s \in (0, 1)$ and a positive integer $M \geq m$.
If for any $n \geq 1$ and for all digits $a_1, \ldots, a_{n-1}$ $(m \leq a_i \leq M, \, i = 1, \ldots, n-1)$ the following inequality holds
\begin{equation}
\left| I_{n-1} (a_1, \ldots, a_{n-1}) \right|^s \geq \sum_{k=m}^{M} \left| I_{n} (a_1, \ldots, a_{n-1}, k) \right|^s, \label{3.11} 
\end{equation}
then $\dim_H (E_M) \leq s$.
\end{lemma}
\begin{proof}
Since $I_0 = [0, \theta]$, by (\ref{3.11}) we obtain 
\[
\theta^s = |I_0|^s \geq \sum_{k_1=m}^{M} |I_1(k_1)|^s 
\geq \sum_{k_1,k_2=m}^{M} |I_2(k_1, k_2)|^s \geq \ldots 
\geq \sum_{k_1,\ldots,k_n=m}^{M} |I_n(k_1, \ldots, k_n)|^s. 
\]
Now let $\delta > 0$ be given. We choose $ n \geq 1$ so large such that the lengths of all $n$-th order intervals are shorter than $\delta$. 
This is possible since $q_n \geq \left\lfloor \displaystyle \frac{n}{2}\right\rfloor \theta^2$ and 
\[
|I_n| = \displaystyle \frac{\theta}{q_n(q_n + \theta q_{n-1})} < \frac{\theta}{q^2_n} \leq \frac{1}{\left\lfloor \displaystyle \frac{n}{2}\right\rfloor^2 \theta^3}.
\]

Since all the $n$-th order intervals $I_n$ form a $\delta$-covering for $E_M$ it follows that for the $s$-covered length of $E_M$ we have 
$H_{\delta, s} (E_M) \leq \sum |I_n|^s \leq |I_0|^s = \theta^s$. 
Consequently 
\[
H_s(E_M) = \lim_{\delta \to 0} H_{\delta, s} (E_M) \leq \theta^s,
\] 
what involves $\dim_H (E_M) \leq s$.

\end{proof}

\begin{proposition} \label{prop.3.4}
For $s = 1 - \displaystyle \frac{m}{M+2} \frac{1}{\log  \displaystyle\frac{2M(M+1)}{m}}$, where $M \geq m$, the Hausdorff dimension $\dim_H(E_M) \leq s$. 
\end{proposition}
\begin{proof}
We have to show that the assumption from Lemma \ref{lema.3.3} is fulfilled. The inequality to be proved is (\ref{3.11}). 
Using (\ref{3.2}) and (\ref{3.3}) we obtain
\begin{equation} \label{3.12}
\frac{\theta^s}{q^s_{n-1} (q_{n-1} + \theta q_{n-2})^s} \geq 
\sum_{k=m}^{M} \frac{\theta^s}{(k \theta q_{n-1} + q_{n-2})^s  \,  ((k+1) \theta q_{n-1} + q_{n-2})^s}.
\end{equation}
We have already shown in (\ref{3.9}) that 
\begin{eqnarray}  
\sum_{k=m}^{M} \frac{1}{(k \theta q_{n-1} + q_{n-2})  \,  ((k+1) \theta q_{n-1} + q_{n-2})} 
&=& \frac{1}{\theta} 
 \frac{1}{q_{n-1}(m \theta q_{n-1} + q_{n-2})} \left( 1 - \frac{m \theta q_{n-1} + q_{n-2}}{(M+1) \theta q_{n-1} + q_{n-2}} \right) \nonumber \\
&=&  \frac{1}{\theta} 
 \frac{\theta}{q_{n-1}(q_{n-1} + \theta q_{n-2})} \left( 1 - \frac{q_{n-1} + \theta q_{n-2}}{(M+1) \theta^2 q_{n-1} + \theta q_{n-2}} \right) \nonumber \\
&<& \frac{1}{q_{n-1}(q_{n-1} + \theta q_{n-2})} \left( 1 - \frac{ 1 }{ \theta^2(M+2) } \right) \nonumber  \\
&=& \frac{1}{q_{n-1}(q_{n-1} + \theta q_{n-2})} \left( 1 - \frac{m}{M+2} \right). \label{3.13}
\end{eqnarray}
Multiplying (\ref{3.13}) with 
$(k \theta q_{n-1} + q_{n-2})^{1-s}  \,  ((k+1) \theta q_{n-1} + q_{n-2})^{1-s}$
we get 
\[
\sum_{k=m}^{M} \frac{(k \theta q_{n-1} + q_{n-2})^{1-s}  \,  ((k+1) \theta q_{n-1} + q_{n-2})^{1-s}}{(k \theta q_{n-1} + q_{n-2})  \,  ((k+1) \theta q_{n-1} + q_{n-2})}
\leq  
\sum_{k=m}^{M} \frac{(M \theta q_{n-1} + q_{n-2})^{1-s}  \,  ((M+1) \theta q_{n-1} + q_{n-2})^{1-s}}{(k \theta q_{n-1} + q_{n-2})  \,  ((k+1) \theta q_{n-1} + q_{n-2})}.
\]
Since $q_{n-2} \leq \theta q_{n-1}$ and $q_{n-2} < M q_{n-2}$ we get 
\begin{eqnarray*} 
(M \theta q_{n-1} + q_{n-2})^{1-s}  \,  ((M+1) \theta q_{n-1} + q_{n-2})^{1-s} 
&<& (M \theta q_{n-1} + M \theta q_{n-1})^{1-s}  \,  ((M+1) \theta q_{n-1} + (M+1)\theta^2 q_{n-2})^{1-s} \\
&=& (2 M \theta q_{n-1})^{1-s} \, ((M+1) \theta)^{1-s} (q_{n-1}+ \theta q_{n-2})^{1-s}.
\end{eqnarray*}
Now by (\ref{3.13}) we get
\begin{eqnarray*}  
\sum_{k=m}^{M} \frac{1}{(k \theta q_{n-1} + q_{n-2})^s  \, ((k+1) \theta q_{n-1} + q_{n-2})^s} 
&<& 
\frac{ (2 M \theta q_{n-1})^{1-s} \, ((M+1) \theta)^{1-s} (q_{n-1}+ \theta q_{n-2})^{1-s} }{q_{n-1}(q_{n-1}+\theta q_{n-2}) } \left( 1 - \frac{m}{M+2} \right)  \\
&=&
\frac{2^{1-s}(M(M+1))^{1-s}\left(\theta^2 \right)^{1-s}}{q^s_{n-1}(q_{n-1}+\theta q_{n-2})^s} \left( 1 - \frac{m}{M+2} \right).
\end{eqnarray*}
Now to obtain (\ref{3.12}) we only have to show that 
\begin{equation} \label{3.14}
\left( \frac{2M(M+1)}{m} \right)^{1-s} \left( 1 - \frac{m}{M+2} \right) \leq 1
\end{equation}
which is equivalent with 
\[
(1-s) \log \frac{2M(M+1)}{m} \leq - \log \left( 1 - \frac{m}{M+2} \right).
\]
Since $x \leq - \log(1-x)$, for all $x \in (0, 1)$, by choosing 
$(1-s) \log \displaystyle \frac{2M(M+1)}{m} = \displaystyle \frac{m}{M+2}$,
and assuming that $\displaystyle \frac{m}{M+2} \in \left(0, 1 \right)$, we get (\ref{3.14}) for $s = 1 - \displaystyle \frac{m}{M+2} \frac{1}{\log \frac{2M(M+1)}{m}}$, with $M \geq m$.
\end{proof}
From Proposition \ref{prop.3.2} and Proposition \ref{prop.3.4} we obtain the following result which significantly strengthens Jarnik's result (see (\ref{3.6})). 

\begin{proposition} \label{Prop.Jarnik}
For any $M > 2m+1$

\begin{equation} \label{3.15}
1 - \displaystyle \frac{2(m+1)}{M+1} \frac{1}{\log(m+1)} \leq \dim_H (E_M) \leq 1 - \displaystyle \frac{m}{M+2} \frac{1}{\log  \displaystyle\frac{2M(M+1)}{m}}. 
\end{equation}
In particular, the set 
\[
E = \left\lbrace x \in \Omega: \sup_{n \geq 1} a_n(x) < + \infty \right\rbrace 
\]
is of Hausdorff dimension $1$. 
\end{proposition}
\begin{remark}
We focus on the case of RCFs when $\theta = 1$, thus for $m = 1 \,\, (\theta^2 = 1/m)$. 
\begin{itemize}
\item[(i)]  Firstly, we obtained less restrictive conditions for $M$ compared to Jarnik.

\item[(ii)] Secondly, our estimates for the lower and upper bounds are much better than those obtained by Jarnik for all values of $M$ ($M > 8$). The Table {4.1} is very suggestive. 

\item[(iii)] In Table {4.2} we also collect some values of the lower and upper bounds for different values of $m \geq 2$ and $M > 2m+1$. 
\end{itemize}
\end{remark}

\begin{table}[h]
\begin{center}
\begin{tabular}{|l|l|l|l|l|}
  \hline
  $M$ & Lower bound (Jarnik) & Lower bound & Upper bound & Upper bound (Jarnik)  \\
  \hline
    & $ 1 - \frac{4}{M \log 2}$ & $ 1 - \frac{4}{M+1}\frac{1}{\log 2}$ & $ 1 - \frac{1}{M+2}\frac{1}{\log 2M(M+1)}$ & $ 1 - \frac{1}{8M \log M}$  \\ 
  \hline
  $9$ & $0.358802204$ & $0.422921983$ & $0.982493771$ & $0.993678894$    \\ 
    \hline
  $100$ & $0.942292198$ & $0.942863562$ & $0.999011047$ & $0.999728565$    \\ 
        \hline
  $10000$ & $0.999422922$ & $0.999422979$ & $0.999994769$ & $0.999998642$    \\ 
          \hline
\end{tabular}
\end{center}
\captionof*{table}{Table 4.1}\label{Table4.1}
\end{table}

\begin{table}[h]
\begin{center}
\begin{tabular}{|l|l|l|l|}
  \hline
  ${m}$ & ${M}$ & {Lower bound} & {Upper bound} \\
  \hline
   $2$ & $7$ & $0.31732058$ & $0.944794333$ \\ 
  \hline
  $100$ & $203$ & $0.785445239$ & $0.927402458$ \\ 
    \hline
  $10000$ & $30000$ & $0.927613546$ & $0.972455324$ \\ 
  \hline
\end{tabular}
\end{center}
\captionof*{table}{Table 4.2}\label{Table4.2}
\end{table}

\section{The main result}

In this section, we prove our second goal of this paper. 
For this, we need the following lemmas.

\begin{lemma} \label{lema.5.1}
$\dim_H(E(0)) = 1$.
\end{lemma}
\begin{proof}
It is clear that $E \subset E(0)$, where $E$ is as in Proposition \ref{Prop.Jarnik}. 
Since $\dim_H(E) = 1$, it follows that $\dim_H(E(0)) = 1$. 
\end{proof}

In the sequel let us study the case: $\eta > 0$. 
Define a sequence $\{n_k\}_{k \geq 1} \subset \mathbb{N}_+$ satisfying $n_k = (k+1)^2$ for any $k \geq 1$. 
For any $M \in \mathbb{N}_m$ and $\eta > 0$ let 
\[
E_M(\eta) =\left\lbrace  x \in \Omega: a_{k^2} (x) = \left\lfloor \frac{\eta k^2}{\log \log k^2} \right\rfloor \mbox{ for all } k \geq 2 \mbox{ and } m \leq a_i(x) \leq M \mbox{ for } i \neq n_k \mbox{ for any } k \geq 1\right\rbrace.
\]
\begin{lemma} \label{lema.5.2}
For any $M \in \mathbb{N}_m$ and $\eta > 0$, $E_M(\eta) \subset E(\eta)$. 
\end{lemma}
\begin{proof}
Choose $k_0$ large enough such that $\left\lfloor \frac{\eta k^2}{\log \log k^2} \right\rfloor \geq M$ and $\frac{k^2}{\log \log k^2}$ is monotone increasing for $k \geq k_0$. 
Fix $x \in E_M(\eta)$, for any $n \geq n_{k_0}$, there exists a positive integer $k \geq k_0$ such that $(k+1)^2 = n_k \leq n < n_{k+1} = (k+2)^2$. 
Thus 
\[
L_n(x) = \max\{a_1(x), \ldots, a_n(x)\} = \max \left\lbrace a_{(k+1)^2}(x), M \right\rbrace = \left\lfloor \frac{\eta \, (k+1)^2}{\log \log (k+1)^2} \right\rfloor. 
\]
Since 
\[
\frac{L_n(x) \log \log n}{n} \geq 
\left\lfloor \frac{\eta \, (k+1)^2}{\log \log (k+1)^2} \right\rfloor \frac{\log \log \left( (k+2)^2 -1\right) }{(k+2)^2-1}
\]
and 
\[
\frac{L_n(x) \log \log n}{n} \leq 
\left\lfloor \frac{\eta \, (k+1)^2}{\log \log (k+1)^2} \right\rfloor \frac{\log \log \left( (k+1)^2 \right) }{(k+1)^2},
\]
we get 
\[
\lim_{n \to \infty} \frac{L_n(x) \log \log n}{n} = \eta. 
\]
\end{proof}

In order to estimate the Hausdorff dimension of $E_M(\eta)$, we shall introduce for any $ n \geq 1$, the sets $A_n$, described as follows:
\begin{eqnarray*}
A_n = \left\lbrace (\alpha_1, \ldots, \alpha_n) \in \mathbb{N}^n_m, \, \alpha_{(k+1)^2} = \left\lfloor \frac{\eta \, (k+1)^2}{\log \log (k+1)^2} \right\rfloor  \mbox{ for any } k \right.  \\
 \left.  \mbox{ satisfying } (k+1)^2 \leq n \mbox{ and } m \leq \alpha_i \leq M \mbox{ for } 1 \leq i \neq n_k \leq n \right\rbrace.
\end{eqnarray*} 

We call $I_0 = [0, \theta]$ the basic interval of order $0$ and for any $ n\geq 1$ and $(\alpha_1, \ldots, \alpha_n) \in A_n$, we call $I_n(\alpha_1, \ldots, \alpha_n)$ the basic interval of order $n$. 
Define 
\begin{equation} \label{5.1}
J_n(\alpha_1, \ldots, \alpha_n) = \bigcup_{\alpha_{n+1}} I_{n+1}(\alpha_1, \ldots, \alpha_n, \alpha_{n+1})
\end{equation}
a fundamental interval of order $n$, where the union is taken over all $\alpha_{n+1}$ such that $(\alpha_1, \ldots, \alpha_n, \alpha_{n+1}) \in A_{n+1}$. 
Consequently
\begin{equation} \label{5.2}
E_M(\eta) = \bigcap_{n \geq 1} \,  \bigcup_{(\alpha_1, \ldots, \alpha_n) \in A_{n}} I_n (\alpha_1, \ldots, \alpha_n) = 
\bigcap_{n \geq 1} \, \bigcup_{(\alpha_1, \ldots, \alpha_n) \in A_{n}} J_n (\alpha_1, \ldots, \alpha_n).
\end{equation}
For any $n \geq 1$, let $c(n) = \#\left\lbrace k \in \mathbb{N}: (k+1)^2 \leq n \right\rbrace $. 
For any $(\alpha_1, \ldots, \alpha_n) \in A_{n}$, let $\overline{(\alpha_1, \ldots, \alpha_n)}$ to be the block by eliminating the terms 
$\left\lbrace \alpha_{n_{k_i}}: 1 \leq i \leq c(n) \right\rbrace$ in $(\alpha_1, \ldots, \alpha_n)$; 
let 
$\overline{[\alpha_1 \theta, \ldots, \alpha_n \theta]}$ be the finite $\theta$-expansion corresponding to $\overline{(\alpha_1, \ldots, \alpha_n)}$. 
Since the length of the block $\overline{(\alpha_1, \ldots, \alpha_n)}$ is $n - c(n)$, let us consider 
\begin{eqnarray*}
\overline{q}_n (\alpha_1, \ldots, \alpha_n) = q_{n-c(n)}\overline{(\alpha_1, \ldots, \alpha_n)}, \\
\overline{I}_n (\alpha_1, \ldots, \alpha_n) = I_{n-c(n)}\overline{(\alpha_1, \ldots, \alpha_n)}.
\end{eqnarray*}
It is easy to see that 
\begin{equation} \label{5.3}
c(n) < \sqrt{n} \mbox{ and } 
\overline{(\alpha_1, \ldots, \alpha_n)} \in \mathbb{N}^{n-c(n)}_{m,M},
\end{equation}
where 
$\mathbb{N}_{m,M}:=\{m, m+1, \ldots, M\}$. 

\begin{lemma} \label{lema.5.3}
For any $0 < \varepsilon < 1$, there exists $N_0 = N_0(\varepsilon)$ such that for any $n \geq N_0$ and 
$(\alpha_1, \ldots, \alpha_n) \in A_n$, we have 
\[
\left| I_n(\alpha_1, \ldots, \alpha_n) \right| 
\geq 
\frac{1}{\left(1+\theta^2\right)\theta^{1+\varepsilon}}
\left| \overline{I}_n(\alpha_1, \ldots, \alpha_n) \right|^{1+\varepsilon}.  
\] 
\end{lemma}
\begin{proof}
For $(\alpha_1, \ldots, \alpha_n) \in A_n$, if $n_k \leq n < n_{k+1}$ we have 
\[
\alpha_{n_k} = \left\lfloor \frac{\eta \, (k+1)^2}{\log \log (k+1)^2} \right\rfloor := \beta_{k+1}. 
\]
Since $\{\beta_k \} _{k \geq 1}$ is an increasing positive integer sequence satisfying $\beta_k \to \infty$ as $k \to \infty$ and $\displaystyle\lim_{k \to \infty} (\beta_{k+1})^{\frac{k+1}{n_k}} = 1$, it follows that 
\[
\lim_{k \to \infty} (\alpha_{n_k}+m)^{\frac{k+1}{n_k}} = 1. 
\]
Thus, there exists an integer $k_0$ such that 
\begin{equation} \label{5.4}
(\alpha_{n_k}+m)^{\frac{k+1}{n_k}} < (m+1)^{\frac{\varepsilon}{4}}
\end{equation}
for any $k \geq k_0$. 
Assume that $n_{k_{c(n)}} \leq n < n_{k_{c(n)}+1}$ for some $k_{c(n)} \geq k_0$ and $k_0$ is also chosen to be sufficiently large such that 
\begin{equation} \label{5.5}
(n_{k_{c(n)}} - k_{c(n)} - 3) \varepsilon > n_{k_{c(n)}} \cdot \frac{\varepsilon}{2}
\end{equation}
which is ensured by $\displaystyle \lim_{k \to \infty} \frac{k+1}{n_k} = 0$.

Then by Lemma \ref{lema.3.01}(i), (\ref{5.3}), (\ref{5.4}) and (\ref{5.5}), it yields that  
\begin{eqnarray} \label{5.6}
\overline{q}^{2\varepsilon}_n &=& q^{2\varepsilon}_{n-c(n)} \geq (m+1)^{(n-c(n)-1)\varepsilon} \geq (m+1)^{(n_{k_{c(n)}} - k_{c(n)} - 3) \varepsilon} \nonumber \\
&\geq& (m+1)^{n_{k_{c(n)}}\cdot \frac{\varepsilon}{2}} > (\alpha_{n_k}+m)^{2(k_{c(n)}+1)} 
> (\alpha_{n_k}+m)^{2c(n)}.
\end{eqnarray}
Take $N_0 = n_{k_0}$. For any $n \geq N_0$, since the sequence $\{ \alpha_{n_k} \}_{k \geq 1}$ is increasing, we have by Lemma \ref{lema.3.01}(ii), (\ref{3.03}) and (\ref{5.6}) 
\begin{eqnarray*}
\left| I_n(\alpha_1, \ldots, \alpha_n) \right| &\geq& \frac{\theta}{\left(1+\theta^2\right)q^2_n(\alpha_1, \ldots, \alpha_n)} \\
&\geq& 
\frac{\theta}{\left(1+\theta^2\right)} \frac{1}{q^2_{n-c(n)}\overline{(\alpha_1, \ldots, \alpha_n)}\left[ \left( \alpha_{n_{k_1}} +m \right) \left( \alpha_{n_{k_2}} +m \right)\cdots \left( \alpha_{n_{k_{c(n)}}} +m \right)  \right]^2\cdot \theta^{2c(n)}} \\
&\geq&
\frac{\theta}{\left(1+\theta^2\right)\theta^{2c(n)}} \frac{1}{q^2_{n-c(n)}\overline{(\alpha_1, \ldots, \alpha_n)} \left( \alpha_{n_{k_{c(n)}}} +m \right)^{2c(n)}} \\
&\geq& 
\frac{1}{\left(1+\theta^2\right)\theta^{2c(n)}} \frac{\theta}{\left( q^2_{n-c(n)}\overline{(\alpha_1, \ldots, \alpha_n)} \right)^{1+ \varepsilon }} \\
&=& 
\frac{1}{\left(1+\theta^2\right)\theta^{2c(n)}} \frac{1}{\theta^{\varepsilon}} \left( \frac{\theta}{ q^2_{n-c(n)}\overline{(\alpha_1, \ldots, \alpha_n)}  } \right)^{1+\varepsilon} 
\geq 
\frac{1}{\left(1+\theta^2\right)\theta^{2c(n)} \theta^{\varepsilon}} \left| \overline{I}_n(\alpha_1, \ldots, \alpha_n) \right|^{1+\varepsilon} \\
&\geq& 
\frac{1}{\left(1+\theta^2\right)\theta \cdot \theta^{\varepsilon}} \left| \overline{I}_n(\alpha_1, \ldots, \alpha_n) \right|^{1+\varepsilon} = 
\frac{1}{\left(1+\theta^2\right)\theta^{1+\varepsilon}} \left| \overline{I}_n(\alpha_1, \ldots, \alpha_n) \right|^{1+\varepsilon}.
\end{eqnarray*}
\end{proof}

Without loss of generality for any two different $x ,y \in E_M(\eta)$, $x < y$, there exists a greatest integer, say $n$, such that $x, y$ are contained in the same basic interval of order $n$. 
Thus there exist $\alpha_1, \ldots, \alpha_n \in \mathbb{N}_m$ and $u_{n+1} \neq v_{n+1}$ such that $(\alpha_1, \ldots, \alpha_n, u_{n+1} ) \in A_{n+1}$, 
$(\alpha_1, \ldots, \alpha_n, v_{n+1} ) \in A_{n+1}$ and $x \in I_{n+1}(\alpha_1, \ldots, \alpha_n, u_{n+1} )$, $y \in I_{n+1}(\alpha_1, \ldots, \alpha_n, v_{n+1} )$
respectively. 
Since 
\begin{eqnarray*}
I_{n+1}(\alpha_1, \ldots, \alpha_n, u_{n+1}) \cap E_M(\eta) &=& J_{n+1}(\alpha_1, \ldots, \alpha_n, u_{n+1}) \cap E_M(\eta), \\
I_{n+1}(\alpha_1, \ldots, \alpha_n, v_{n+1}) \cap E_M(\eta) &=& J_{n+1}(\alpha_1, \ldots, \alpha_n, v_{n+1}) \cap E_M(\eta),
\end{eqnarray*}
we have 
\[
x \in J_{n+1}(\alpha_1, \ldots, \alpha_n, u_{n+1}), \quad 
x \in J_{n+1}(\alpha_1, \ldots, \alpha_n, v_{n+1}).
\]
Consequently, $y-x$ is greater than or equal to the gap between $J_{n+1}(\alpha_1, \ldots, \alpha_n, u_{n+1})$ and $J_{n+1}(\alpha_1, \ldots, \alpha_n, v_{n+1})$. 
Notice that $n+1 \neq n_k$ for any $k \in \mathbb{N}$, i.e., $m \leq u_{n+1} \neq v_{n+1} \leq M$, otherwise $u_{n+1} = v_{n+1} = \left\lfloor \frac{\eta \, (k+1)^2}{\log \log (k+1)^2} \right\rfloor$ for some $k \geq 1$, which contradicts $u_{n+1} \neq v_{n+1}$. 
\begin{lemma} \label{lema.5.4}
$y - x \geq \mathcal{K}(\theta, M) \cdot \left| I_n(\alpha_1, \ldots, \alpha_n) \right|$, 
where 
\[
\mathcal{K}(\theta, M) := \frac{ m }
{ \theta \left( M(M+1) +(M+1)\theta +m \right) \left( M + \theta +1 \right) }. 
\] 
\end{lemma}
\begin{proof}
We assume that $n$ is even. 
We can proceed in the same way when
$n$ is odd.
The proof is divided into two parts.

$\left. \mathrm{I}\right)$ 
$n+2 = n_k$ for some $k \geq 1$. 

By (\ref{3.1}), $y-x$ is greater than the distance between $J_{n+1}(\alpha_1, \ldots, \alpha_n, u_{n+1})$'s right end point 
$\left[ \alpha_1 \theta, \ldots, \alpha_n \theta, u_{n+1}\theta, \left( \left\lfloor \frac{\eta \, (k+1)^2}{\log \log (k+1)^2} \right\rfloor+1 \right) \theta \right] $
and 
$J_{n+1}(\alpha_1, \ldots, \alpha_n, v_{n+1})$'s left end point 
$\left[ \alpha_1 \theta, \ldots, \alpha_n \theta, v_{n+1}\theta, \left\lfloor \frac{\eta \, (k+1)^2}{\log \log (k+1)^2} \right\rfloor  \theta \right]$.
Thus we obtain 
\begin{linenomath*}
\begin{eqnarray} \label{5.7}
y - x 
&\geq& 
\left| 
\frac{ \left( u_{n+1}\theta + \frac{1}{\left( \left\lfloor \frac{\eta \, (k+1)^2}{\log \log (k+1)^2} \right\rfloor+1 \right) \theta}\right)p_n + \theta p_{n-1}}
{\left( u_{n+1}\theta + \frac{1}{\left( \left\lfloor \frac{\eta \, (k+1)^2}{\log \log (k+1)^2} \right\rfloor+1 \right) \theta}\right)q_n + \theta q_{n-1} } 
-
\frac{ \left( v_{n+1}\theta + \frac{1}{ \left\lfloor \frac{\eta \, (k+1)^2}{\log \log (k+1)^2} \right\rfloor  \theta}\right)p_n + \theta p_{n-1}}
{\left( v_{n+1}\theta + \frac{1}{ \left\lfloor \frac{\eta \, (k+1)^2}{\log \log (k+1)^2} \right\rfloor \theta}\right)q_n + \theta q_{n-1} }
\right| \nonumber \\
&\geq& 
\frac{ \theta \left|  (u_{n+1} - v_{n+1})\theta  + \frac{1}{ \left( \left\lfloor \frac{\eta \, (k+1)^2}{\log \log (k+1)^2} \right\rfloor +1\right)\theta } - \frac{1}{ \left\lfloor \frac{\eta \, (k+1)^2}{\log \log (k+1)^2} \right\rfloor \theta}  \right| }{ \left(\left( u_{n+1}\theta + \frac{1}{\left( \left\lfloor \frac{\eta \, (k+1)^2}{\log \log (k+1)^2} \right\rfloor+1 \right) \theta}\right)q_n + \theta^2 q_{n} \right) \left(\left( v_{n+1}\theta + \frac{1}{ \left\lfloor \frac{\eta \, (k+1)^2}{\log \log (k+1)^2} \right\rfloor \theta}\right)q_n + \theta^2 q_{n} \right) }\nonumber \\
&\geq& 
\frac{ \theta^2 \frac{M}{M+1} }{q^2_n \left(  M\theta + \frac{1}{\left( M+1 \right) \theta}  + \theta^2  \right) \left(  M\theta + \frac{1}{ M \theta} + \theta^2 \right) } \nonumber \\
&\geq& \frac{\theta}{q^2_n} \cdot \frac{M^2 \theta^3}{\left( M(M+1)\theta^2 +(M+1) \theta^3 +1\right)\left(M^2 \theta^2 +M \theta^3 +1 \right) } \nonumber \\
&=&
\frac{\theta}{q^2_n} \cdot \frac{M^2}
{\theta \left( M(M+1)+(M+1) \theta + \frac{1}{\theta^2} \right)\left(M^2  +M \theta +  \frac{1}{\theta^2} \right) } \nonumber \\
&\geq& \frac{M^2}
{\theta \left( M(M+1)+(M+1) \theta + m \right)\left(M^2  +M \theta +  m \right) }
  \left| I_n(\alpha_1, \ldots, \alpha_n) \right|.
\end{eqnarray}
\end{linenomath*}

$\left. \mathrm{II}\right)$ 
$n+2 \neq n_k$ for any $k \geq 1$.

Similarly to part I, we have
\begin{linenomath*}
\begin{eqnarray} \label{5.8}
y-x
&\geq& 
\left| 
\frac{ \left( u_{n+1}\theta + \frac{1}{\left( M+1 \right) \theta}\right)p_n + \theta p_{n-1}}
{\left( u_{n+1}\theta + \frac{1}{\left( M+1 \right) \theta}\right)q_n + \theta q_{n-1} } 
-
\frac{ \left( v_{n+1}\theta + \frac{1}{ m  \theta}\right)p_n + \theta p_{n-1}}
{\left( v_{n+1}\theta + \frac{1}{ m \theta}\right)q_n + \theta q_{n-1} }
\right| \nonumber \\
&\geq&
\frac{ \theta \left| (u_{n+1} -v_{n+1}) \theta + \frac{1}{(M+1)\theta} -\frac{1}{m \theta} \right|  }
{ \left( \left( u_{n+1}\theta + \frac{1}{\left( M+1 \right) \theta}\right)q_n + \theta q_{n-1} \right) \left( \left( v_{n+1}\theta + \frac{1}{ m \theta}\right)q_n + \theta q_{n-1} \right)} \nonumber \\
&\geq&
\frac{ \theta \frac{\theta m}{M+1}  }
{ \left( \left( M\theta + \frac{1}{\left( M+1 \right) \theta}\right)q_n + \theta^2 q_{n} \right) \left( \left( M\theta + \frac{1}{ m \theta}\right)q_n + \theta^2 q_{n} \right)} \nonumber \\
&=&
\frac{ 1 }
{q^2_n (M+1)\left(  M\theta + \frac{1}{\left( M+1 \right) \theta}  + \theta^2 \right) \left(   M\theta + \frac{1}{ m \theta} + \theta^2  \right)} \nonumber \\
&=&
\frac{\theta}{q^2_n} \cdot
\frac{ m \theta}
{ \left( M(M+1)\theta^2 +(M+1)\theta^3 +1 \right) \left( Mm\theta^2 +m\theta^3 +1 \right) } \nonumber \\
&=&
\frac{\theta}{q^2_n} \cdot
\frac{ m \theta}
{ \theta^2 \left( M(M+1) +(M+1)\theta +m \right) \left( M + \theta +1 \right) } \nonumber \\
&\geq&  
\frac{ m }
{ \theta \left( M(M+1) +(M+1)\theta +m \right) \left( M + \theta +1 \right) }
\left| I_n(\alpha_1, \ldots, \alpha_n) \right|.
\end{eqnarray}
\end{linenomath*}
Thus, from (\ref{5.7}) and (\ref{5.8}), the proof is complete. 
\end{proof}

For a fixed $\eta > 0$ and $M > 2m+1$, consider the map $f: E_M(\eta) \to E_M$ defined by 
\[
f(x) = \lim_{n \to \infty} \overline{[\alpha_1 \theta, \ldots, \alpha_n \theta]}
\]
for any $x = [\alpha_1 \theta, \alpha_2 \theta, \ldots, \alpha_n \theta, \ldots] \in E_M(\eta)$.

For any $ 0 < \varepsilon < 1$, by Lemma \ref{lema.5.3} and based on the estimation given by Lemma \ref{lema.5.4}, let us define another constant 
\[ 
\mathcal{K}_1(\theta, M) = \frac{\mathcal{K}(\theta, M)}{1+\theta^2} \cdot \min_{(\alpha_1, \ldots, \alpha_{N_0}) \in A_{N_0}}  \left\lbrace  \left| I_{N_0}(\alpha_1, \ldots, \alpha_{N_0})  \right| \right\rbrace,
\]
where $N_0$ is as in Lemma \ref{lema.5.3}. 
Now, by Lemma \ref{lema.5.3} and Lemma \ref{lema.5.4}, for any two different numbers $x, y \in E_M(\eta)$ we have that 
\begin{eqnarray*}
\left| f(x) - f(y) \right| 
&\leq&
\left| \overline{I}_n(\alpha_1, \ldots, \alpha_n) \right| 
\leq
\theta \left( 1+\theta^2 \right)^{\frac{1}{1+\varepsilon}} 
\left| I_n(\alpha_1, \ldots, \alpha_n) \right|^{\frac{1}{1+\varepsilon}} \\
&\leq& 
\theta \left( 1+\theta^2 \right)^{\frac{1}{1+\varepsilon}}
\cdot
\frac{| x- y |^{\frac{1}{1+\varepsilon}}}{ \left(\mathcal{K}(\theta, M) \right)^{\frac{1}{1+\varepsilon}} } 
\leq 
\left( \frac{ 1+\theta^2 }{ \mathcal{K}(\theta, M) } \right)^{\frac{1}{1+\varepsilon}} | x- y |^{\frac{1}{1+\varepsilon}} \\
&\leq& 
\frac{ 1+\theta^2 }{ \mathcal{K}(\theta, M) } | x- y |^{\frac{1}{1+\varepsilon}},  
\end{eqnarray*}
if $|x-y|< \mathcal{K}_1(\theta, M)$.
It means that the function $f$ is $\frac{1}{1+\varepsilon}$-H\"{o}lder on all intervals whose lengths are bounded by $\mathcal{K}_1(\theta, M)$. 
By Lemma \ref{lema.2.1} for $E_M(\eta) \cap I$, where $I$ is an arbitrary interval such that $|I| \leq \mathcal{K}_1(\theta, M)$, we have 
\[
\dim_H \left( f(E_M(\eta)) \cap I \right) \leq 
(1+\varepsilon) \dim_H \left( E_M(\eta) \cap I \right). 
\]
It follows immediately that 
\begin{equation} \label{5.9}
 \dim_H \left( f(E_M(\eta)) \right) \leq \dim_H \left( E_M(\eta)\right) 
\end{equation} 
since $\varepsilon$ and $I$ are arbitrary. 

Next we will show that $f\left(E_M(\eta) \right) = E_M$. 
Clearly, $f\left(E_M(\eta) \right) \subset E_M$. 
For each $y = [\alpha_1 \theta, \alpha_2 \theta, \ldots] \in E_M$, we can construct the inverse image of $y$ by inserting in $[\alpha_1 \theta, \alpha_2 \theta, \ldots]$ a sequence of big partial quotients $\{\beta_k\}_{k\geq 2}$. 
Actually, we put 
\[
f^{-1}(y) = 
[\alpha_1 \theta, \alpha_2 \theta, \alpha_3 \theta, \beta_2 \theta, \alpha_4 \theta, \ldots, \alpha_8 \theta, \beta_3 \theta, \alpha_9 \theta, \ldots, \alpha_{k^2-1} \theta, \beta_k \theta, \alpha_{k^2} \theta, \ldots],
\]
where 
$\beta_{k}= \left\lfloor \frac{\eta \, k^2}{\log \log k^2} \right\rfloor$. 
Then $f^{-1}(y) \in E_M(\eta)$. 
By (\ref{5.9}) and Proposition \ref{Prop.Jarnik}, we have
\[
\dim_H (E_M(\eta)) \geq \dim_H(E_M)  \geq   1 - \displaystyle \frac{2(m+1)}{M+1} \frac{1}{\log(m+1)}, 
\]
for any $M > 2m+1$. 
From Lemma \ref{lema.5.2}, we have 
\[
\dim_H(E(\eta)) \geq  1 - \displaystyle \frac{2(m+1)}{M+1} \frac{1}{\log(m+1)}. 
\]
Since $M > 2m+1$ is arbitrary, we have 
$\dim_H (E(\eta)) = 1$. 

Therefore, we may state the main result of this section. 

\begin{theorem} \label{th.5.5}
Let $\eta \geq 0$ and $E(\eta)$ be as in (\ref{1.9}). Then
\[
\dim_H (E(\eta)) = 1.
\]
\end{theorem}



\end{document}